\documentclass[11pt,reqno]{amsart}
\usepackage{amsmath}
\usepackage{amsthm}
\usepackage{amssymb}
\usepackage{enumerate}
\usepackage{amscd}
\usepackage{color}
\usepackage{pb-diagram}
\usepackage{graphicx}
\usepackage[all, cmtip]{xy}
\usepackage{soul}
\usepackage{esint}
\usepackage{hyperref}
\hypersetup{pdftex,colorlinks=true,allcolors=blue}
\usepackage{hypcap}
\usepackage[titletoc]{appendix}
\usepackage{comment}

\usepackage{amsrefs}

\theoremstyle{plain}
\newtheorem{lemma}{Lemma}
\newtheorem{prop}[lemma]{Proposition}
\newtheorem{coro}[lemma]{Corollary}
\newtheorem{theo}[lemma]{Theorem}

\newtheorem{defi}[lemma]{Definition}
\newtheorem{thm-Intro}{Theorem} 
\newtheorem{cor-Intro}{Corollary} 

\newtheorem{ques}[lemma]{Question}
\numberwithin{equation}{section}

\newcommand{\Hol}{\textup{Hol}}

\begin{document}
\title[Non-existence of complete K\"ahler metric]{Non-existence of complete K\"ahler metric of negatively pinched holomorphic sectional curvature}

\author[Gunhee Cho]{Gunhee Cho}
\thanks{\footnotemark  This research was supported in part by a Simons Travel Grant.}
\address{Department of Mathematics\\
	University of California, Santa Barbara\\
	552 University Rd, Isla Vista, CA 93117.}
\email{gunhee.cho@math.ucsb.edu}


\begin{abstract} 
We prove a theorem which provides a sufficient condition for the non-existence of a complete K\"ahler--Einstein metric of negative scalar curvature whose holomorphic sectional curvature is negatively pinched:

Let $\Omega$ be a bounded weakly pseudoconvex domain in $\mathbb{C}^n$ with a K\"ahler metric $\omega$ whose holomorphic sectional curvature is negative near the topological boundary of $\Omega$ (with respect to the relative topology of $\mathbb{C}^n$) and $\omega$ admits quasi-bounded geometry. Then $\omega$ is uniformly equivalent to the Kobayashi--Royden metric and the following dichotomy holds:
	
	1. $\omega$ is complete, and $\omega$ is uniformly equivalent to the complete K\"ahler--Einstein metric with negative scalar curvature. 
	
	2. $\omega$ is incomplete, and there is no complete K\"ahler metric with negatively pinched holomorphic sectional curvature. Moreover, $\Omega$ is Carath\'eodory incomplete. 
	
Our approach is based on the construction of a K\"ahler metric of negatively pinched holomorphic sectional curvature and applying the implication of equivalence of invariant metrics inspired by Wu-Yau. 	
\end{abstract}

\maketitle


\section{Introduction}

This following question has been addressed very recently in relation to anti-bisectional curvature in a paper by Khan-Zheng \cite{GKFZ20}:

\begin{ques}
	When is it the case that a K\"ahler manifold does not admit a K\"ahler-Einstein metric with negative holomorphic sectional curvature?
\end{ques}

To address this question, it is reasonable to consider a bounded weakly pseudoconvex domain in $\mathbb{C}^n$ because pseudoconvexity is a necessary and sufficient condition for existence of complete K\"ahler-Einstein metric of negative scalar curvature on bounded domains in $\mathbb{C}^n$ by Mok-Yau \cite{NMY83}. In this paper, we show the following: 

\begin{theo}\label{thm:general_statement}
	Let $\Omega$ be a bounded weakly pseudoconvex domain in $\mathbb{C}^n$ with a K\"ahler metric $\omega$ whose holomorphic sectional curvature is negatively pinched near the topological boundary of $\Omega$ (with respect to relative topology of $\mathbb{C}^n$) and $\omega$ admits quasi-bounded geometry. Then $\omega$ is uniformly equivalent to the Kobayashi--Royden metric and the following dichotomy holds:
	
	1. $\omega$ is complete, and $\omega$ is uniformly equivalent to the complete K\"ahler--Einstein metric with negative scalar curvature, or
	
	2. $\omega$ is incomplete, and there is no complete K\"ahler metric with negatively pinched holomorphic sectional curvature. Moreover, $\Omega$ is Carath\'eodory incomplete. 
\end{theo}

Our approach is by constructing one K\"ahler metric of negatively pinched holomorphic sectional curvature on $\Omega$ (not just near the boundary) and applying its consequences of equivalence of invariant metrics based on Wu-Yau's work \cite{DY17}. The dichotomy happens because the completeness of the K\"ahler metric $\omega$ is not required in the process of construction of the K\"ahler metric of negatively pinched holomorphic sectional curvature in Theorem~\ref{thm:general_statement}. 

For classes of bounded domains in $\mathbb{C}^n$ with uniform squeezing properties which must be weakly pseudoconvex, the quasi-bounded geometry of the Bergman metric and the complete K\"ahler--Einstein metric of negative scalar curvature are  necessary conditions by work of Yeung \cite[Theorem 2-(b)]{SKY09}. Also, it is known that holomorphic sectional curvature of the Bergman metric is negative near the boundary for smoothly bounded strictly pseudoconvex domains \cite{KLE78}, and for the symmetrized bidisc by recent work of Cho-Yuan \cite{GCYY20} (also see \cite{MR3765515}). On the other hand, there exists an example of a bounded domain in $\mathbb{C}$ with an incomplete Bergman metric, and one can take the product of this domain for higher dimensional cases \cite{MR1841396}. Also, there are several known bounded weakly pseudoconvex domains that are not Kobayashi complete which implies Carath\'eodory incomplete, refer to \cite[Chapter 14]{JP13}.

Theorem~\ref{thm:general_statement} under the incompleteness of $\omega$ implies the following: 

\begin{coro}
	Let $\Omega$ be a bounded weakly pseudoconvex domain in $\mathbb{C}^n$ with an incomplete K\"ahler metric $\omega$ satisfying all assumptions in Theorem~\ref{thm:general_statement}.	Then there is no complete K\"ahler metric with negatively pinched holomorphic sectional curvature. In particular, any complete K\"ahler--Einstein metric of negative scalar curvature does not admit a negatively pinched holomorphic sectional curvature.  
\end{coro}

Theorem~\ref{thm:general_statement} under the completeness of $\omega$ implies the equivalence of two invariant metrics on closed submanifolds of $\Omega$:

\begin{coro}\label{cor:new_class}
	Given a bounded weakly pseudoconvex domain $\Omega$ in $\mathbb{C}^n$ with a complete K\"ahler metric $\omega$ satisfying all assumptions in Theorem~\ref{thm:general_statement}, any closed complex submanifold $S$ of $\Omega$ admits a complete K\"ahler metric $g$ with holomorphic sectional curvature negatively pinched. Moreover, $g$, the unique complete K\"ahler-Einstein metric $g_S^{KE}$ of negative scalar curvature, and the Kobayashi-Royden metric $g_S^K$ are uniformly equivalent.
\end{coro} 

Theorem~\ref{thm:general_statement} also has implications for quotient spaces. For quotient spaces, we define a lattice $\Gamma$ on $\Omega$ to be a discrete group acting properly discontinuously as biholomorphisms on $\Omega$. 

\begin{coro}\label{cor:new_class2}
	Given a bounded weakly pseudoconvex domain $\Omega$ in $\mathbb{C}^n$ with a complete K\"ahler metric $\omega$ satisfying all assumptions in Theorem~\ref{thm:general_statement}, assume that $\Gamma$ is a torsion-free lattice on $\Omega$. Then a compact quotient $N=\Omega/\Gamma$ has to be a projective algebraic variety of general type. A non-compact quotient which has finite volume with respect to $\omega$ has to be a quasi-projective variety of log-general type.
\end{coro}

\subsection*{The Acknowledgments}
Author thanks professor Damin Wu and professor Kwu-Hwan Lee for very helpful discussions. The author thanks the anonymous referee for a careful reading and remarks that greatly helped improve the presentation of the paper. 




\section{Invariant metrics, Quasi-bounded geometry, Holomorphic sectional curvature}
Let $\Omega$ be a domain in $\mathbb{C}^n$. A pseudometric $F(z,u) : \Omega \times \mathbb{C}^n \rightarrow [0,\infty]$ on a domain $\Omega$ in $\mathbb{C}^n$ is called (biholomorphically) invariant if $F(z,\lambda u)=|\lambda|F(z,u)$ for all $\lambda \in \mathbb{C}^n$, and $F(z,u)=F(f(z),f’(z)u)$ for any biholomorphism $f : \Omega \rightarrow \Omega’$. The Kobayashi-Royden metric, Bergman metric and a K\"ahler-Einstein metric of negative scalar curvature are examples of invariant metrics on bounded weakly pseudoconvex domains in $\mathbb{C}^n$. 

For a bounded domain $\Omega$ in $\mathbb{C}^n$, denote by $A^2(\Omega)$ the holomorphic functions in $L^2(\Omega)$. Let $\{\varphi_j : j \in \mathbb{N}  \}$ be an orthonormal basis for $A^2(\Omega)$ with respect to the $L^2$-inner product. The Bergman kernel $K_\Omega$ associated to $\Omega$ is given by 
\begin{equation*}
	K_{\Omega}(z,\overline{z})=\sum_{j=1}^{\infty}\varphi_j(z)\overline{\varphi_j(z)}. 
\end{equation*}
Note that $K_{\Omega}$ does not depend on the choice of orthonormal basis, and gives rise to an invariant metric, the Bergman metric on $\Omega$:
\begin{equation*}
	g^B_{\Omega}(\xi,\xi)={\sum_{\alpha,\beta =1}^{n}\frac{{\partial}^2 \log K_{\Omega}(z,\overline{z}) }{\partial z_{\alpha}\partial \overline{z_{\beta}}}\xi_{\alpha}\overline{\xi_{\beta}}}.
\end{equation*}

The fundamental property of the Bergman kernel $K_{\Omega}$ on a bounded domain $\Omega \subset \mathbb{C}^n$ is the following transformation law under a biholomorphic map $\Psi : \Omega_1 \rightarrow \Omega_2$ between two bounded domains $\Omega_1,\Omega_2$ in $\mathbb{C}^n$:
\begin{equation*}
	K_{\Omega_2}\circ \Psi = |det \Psi'|^{-2} K_{\Omega_1},
\end{equation*}
here $det \Psi'$ is the holomorphic Jacobian of $\Psi$. This transformation law implies the biholomorphic invariance of the Bergman metric, i.e., the invariant metric. 

Let $\mathbb{D}$ denote the open unit disk in $\mathbb{C}$. Let $z\in \Omega$ and $v\in T_{z}\Omega$ a tangent vector at $z$. The Kobayashi-Royden metric is defined by 
\begin{equation}
	\chi_{\Omega}(z;v)=\inf\{\frac{1}{\alpha} : \alpha>0, f\in \Hol(\mathbb{D},\Omega), f(0)=z, f'(0)=\alpha v\}.
\end{equation}
The Carath\'eodory--Reiffen metric is defined by
\[\gamma_{\Omega}(z;v)=\sup\{|df(z)v| : f\in \Hol(\Omega,\mathbb{D})  \}. \]

Note that each of these three invariant metrics are also defined on any complex manifolds.

The existence of a complete K\"ahler-Einstein metric on a bounded pseudoconvex domain was given in the main theorem in \cite{NMY83}. Based on this result, we can always find a unique complete K\"ahler-Einstein metric of Ricci curvature $-1$ on a bounded weakly pseudoconvex domain $G$, i.e., $g^{KE}_{\Omega}$ satisfies $g^{KE}_{\Omega}=-Ric_{g^{KE}_{\Omega}}$ as a two tensor.

The notion of quasi-bounded geometry was introduced by S.T. Yau and S.~Y. Cheng (\cite{CY80}) and we follow the description of quasi-bounded geometry from Wu and Yau \cite{DY17}. We adopt the following formulation. Let $(M,\omega)$ be an $n$-dimensional K\"ahler manifold. Denote by $B_{\mathbb{C}^n}(r)$ the open ball centered at the origin in $\mathbb{C}^n$ of radius $r$ with respect to the standard metric $\omega_{\mathbb{C}^n}$. 

\begin{defi}\label{def_quasi_bounded}
	An $n$-dimensional K\"ahler manifold $(M,\omega)$ is said to have {\em quasi-bounded geometry} if there exist two constants $r_2>r_1>0$  such that for each point $p\in M$, there is a domain $U\subset \mathbb{C}^n$ and a nonsingular holomorphic map $\psi : U \rightarrow M$ satisfying the following properties:
	
	(1) $B_{\mathbb{C}^n}(r_1)\subset U \subset B_{\mathbb{C}^n}(r_2)$ and $\psi(0)=p$;
	
	(2) there exists a constant $C>0$ depending only on $r_1,r_2,n$ such that 
	\begin{equation*}
	C^{-1}\omega_{\mathbb{C}^n} \leq \psi^{*}(\omega) \leq C\omega_{\mathbb{C}^n} \quad \text{ on } U;
	\end{equation*}
	
	(3)	for each integer $l\geq 0$, there exists a constant $A_l$ depending only on $l,n,r_1,r_2$ such that 
	\begin{equation*}
	\sup_{x\in U}\left |\frac{\partial^{|\nu|+|\mu|}g_{i\overline{j}} }{\partial v^{\mu}\, \partial \overline{v}^{\nu} } \right |\leq A_l, \text{ for all } |\mu|+|\nu|\leq l, 
	\end{equation*}
	where $g_{i\overline{j}}$ are the components of $\psi^{*}\omega$ on $U$ in terms of the natural coordinates $(v^1,\cdots,v^n)$, and $\mu,\nu$ are multiple indices with $|\mu|=\mu_1+\cdots +\mu_n$.
	We choose $r_1$ the largest possible number for a fixed K\"ahler metric $\omega$ and call it the {\em radius} of quasi-bounded geometry. 
\end{defi}

In Definition~\ref{def_quasi_bounded}, the larger radius $r_2$ is needed in the interior Schauder estimates for the Monge-Amp\`ere equation in \cite{DY17}: Those estimates hold on bounded domains with a uniform upper bound for diameters. For the radius of quasi-bounded geometry $r_1$, as the injectivity radius in Riemannian geometry is defined for a fixed Riemannian metric, $r_1$ is also defined for a fixed K\"ahler metric $\omega$ and a fixed quasi-coordinate atlas $\{(U, \psi)\}$ on $M$. The latter is implicitly indicated in \cite[Definition 8]{DY17}. In this paper we rigorously state that $r_1> 0$ is a radius of quasi-bounded geometry of $(M, \omega)$ associated with the quasi-coordinate atlas on $M$ and define `the' radius of quasi-bounded geometry of $(M, \omega)$ to be the largest possible number of the above $r_1$'s. 

Comparing Riemannian geometry, if the injectivity radius of a Riemannian manifold is nonzero, one can rescale it to be arbitrarily large, but such a rescaling usually does not make essential change for a geometric problem. To carry out geometric analysis on manifolds, especially to handle the subtle situation in the case of Riemannian injectivity radius being zero, what we need is \emph{one} $r_1 > 0$ or \emph{one pair} $r_2 > r_1 > 0$ in \cite[Definition 8]{DY17} that depend only on the curvature bounds and dimension (\cite[Theorem 9]{DY17}). On the other hand, the injective radius is the radius of the geodesic ball for which the exponential map is bijective onto its image, whereas the non-singular holomorphic maps required by the radius of quasi-bounded geometry need not be bijective onto its image. 

The following comparison inequality is proven in the Wu-Yau paper. Those authors assume that the K\"ahler metric is complete, but their proof does not actually use completeness.

\begin{lemma}\cite[Lemma 20]{DY17}\label{q-bounded}
Suppose a K\"ahler manifold $(M,\omega)$ has quasi-bounded geometry. Then the Kobayashi-Royden metric $\chi_{M}$ satisfies
\begin{equation*}
	\chi_{M}(x;\xi)\leq C|\xi|_\omega, \text{ for all } x\in M, \xi\in T'_x M,
\end{equation*}
where $C$ depends only on the radius of quasi-bounded geometry of $(M,\omega)$. 
\end{lemma}

A lower bound on the Kobayashi-Royden metric can be obtained from the following Lemma.
\begin{lemma}\cite[Lemma 19]{DY17}\label{ub_hsc}
Let $(M,\omega)$ be a hermitian manifold such that the holomorphic sectional curvature $H_\omega \leq -\kappa<0$. Then,
	\begin{equation*}
		\chi_{M}(x;\xi)\geq \sqrt{\frac{\kappa}{2}} |\xi|_\omega, \text{ for all } x\in M, \xi\in T'_x M,
	\end{equation*}
\end{lemma}

The following Lemma reduces the holomorphic sectional curvature by analyzing the Gaussian curvature.

\begin{lemma}\cite[Lemma 4]{WU73}\label{lemma:wu}
	Let $M$ be a hermitian manifold with hermitian metric $g$, and $t$ be a unit tangent vector to $M$ at $p$. Then there exists an imbedded $1$-dimensional complex submanifold $M'$ of $M$ tangent to $t$ such that the Gaussian curvature of $M'$ at $p$ relative to the induced metric equals the holomorphic sectional curvature $H_g(t)$ of $t$ assigned by $g$. 
\end{lemma}

For each positive smooth function $g(z,\overline{z})$ defined in an open set in $\mathbb{C}$, define a real-valued function $H(g)$ in the same open set as follows:
\begin{equation}\label{eq:gaussian_curv}
	H(g)=-\frac{1}{g}\frac{\partial^2 \log g}{\partial z{\partial} \overline z}.
\end{equation}
$H(g)$ is the Gaussian curvature of the metric $ds^2=2gdz d\overline{z}$. 

\begin{prop}\cite[p19, Proposition 3.1]{K70}\label{prop:HSC}
	For positive smooth functions $f$ and $g$,
	
	(a) $cH(cg)=H(g)$ for all positive numbers $c$;
	
	(b) $fg H(fg)=fH(f)+gH(g)$;
	
	(c) $(f+g)^2 H(f+g)\leq f^2H(f)+g^2 H(g)$
	
	(d) If $H(f)\leq -k_1<0$ and $H(g)\leq -k_2<0$, then 
	\begin{equation*}
		H(f+g)\leq \frac{-k_1k_2}{k_1+k_2}.
	\end{equation*}
	
\end{prop}

\section{Proof of Theorem~\ref{thm:general_statement}}

The following proposition is very slightly different from the form written without proof in  \cite[p 280]{KLE78}. 

\begin{prop}\label{prop:hsc-ub}
	Let $\Omega$ be a bounded weakly pseudoconvex domain in $\mathbb{C}^n$ with hermitian metrics $h$ and $g$. Let $m \in \mathbb{N}$ and consider a hermitian metric $g_\Omega:=m h+g$. Then	
	\begin{equation}\label{eq:estimate_hsc}
		H_{g_\Omega}\leq \left(\frac{1}{1+mu}\right)^2\|H_{g}\|_{\Omega}+\left(\frac{u}{1/m+u} \right)\frac{1}{m}H_{h},
	\end{equation}
	where $u=\min h(t,t)$ for all $t\in \mathbb{C}^n$ with $g(t,t)=1$. 
	\begin{proof}
		Let $t'$ be a unit vector at $p\in \Omega$ relative to $g_\Omega=m h+g$. By Lemma~\ref{lemma:wu}, there exists a $1$-dimensional imbedded complex submanifold $\Omega'$ tangent to $t'$ such that 
\begin{equation*}
	H(mh'+g')(p)=H(mh+g)(t'),
\end{equation*}		
where $g'$ and $h'$ are respectively the induced metrics of $g$ and $h$ on $\Omega$, and therefore $g'+mh'$ is the induced metric of $g+mh$ on $\Omega'$. Thus we can regard $g$ and $h$ as positive smooth functions defined in an open set in $\mathbb{C}$,  the holomorphic sectional curvatures $H(mh+g)$ of a hermitian metric $mh+g$ can be regarded as ~\eqref{eq:gaussian_curv}. Then for each $m\in \mathbb{N}$, by Proposition~\ref{prop:HSC},
		\begin{align*}
			H_{g_\Omega}=H(mh+g)&\leq \left(\frac{mh}{mh+g} \right)^2H(mh)+\left(\frac{g}{mh+g} \right)^2H(g)\\
			&=\left(\frac{mh}{mh+g} \right)^2 \frac{1}{m} H(h)+\left(\frac{g}{mh+g} \right)^2H(g)\\
			&=\left(\frac{u}{u+\frac{1}{m}} \right)^2 \frac{1}{m} H(h)+\left(\frac{1}{1+mu} \right)^2H(g)\\
			&\leq \left(\frac{1}{1+mu} \right)^2||H_g||_{\Omega}+\left(\frac{u}{u+\frac{1}{m}} \right)^2 \frac{1}{m} H_h,
		\end{align*}	
		where $u=\min h(t,t)$ for all $t\in \mathbb{C}^n$ with $g(t,t)=1$. Here $||H_g||_{\Omega}$ means the supremum norm of $H_g(t)$ (also see p.280 in \cite{KLE78}).	
	\end{proof}

\end{prop}

In the following proof, we refer $\omega$ and $g$ refer the same metric in one case viewed as a K\"ahler two-form and in the other case viewed as a metric tensor.

\begin{proof}[Proof of Theorem~\ref{thm:general_statement}]
	Consider a K\"ahler metric of the form 
	\begin{equation}\label{eq:Kahler-metric}
		\omega_\Omega:=m \omega_{P}+\omega, \quad m>0,
	\end{equation}		
	where $\omega_{P}$ is the Poincar\'e metric as a two-form of a ball $D$ in $\mathbb{C}^n$ with $\Omega \Subset D$.  Since $\omega_{P}$ restricted to $\Omega$ is incomplete, $\omega_\Omega$ is complete on $\Omega$ for each $m>0$ if and only if $\omega$ is complete.  
	
	For $\epsilon>0$, let $\mathcal C_\epsilon:=\{ z \in  \Omega  :\, H_{\omega}>-\epsilon \}$. Then it follows from Proposition~\ref{prop:hsc-ub}, on the set $\mathcal C_\epsilon$, one has
	\begin{equation}\label{eq:estimate_hsc}
		H_{\omega_\Omega}\leq \left(\frac{1}{1+mu}\right)^2\|H_{\omega}\|_{\Omega}+\left(\frac{u}{1/m+u} \right)\frac{1}{m}H_{\omega_{P}},
	\end{equation}

	where $u=\min g_{P}(t,t)$ for all $t\in \mathbb{C}^n$ with $g_B(t,t)=1$. Since the sum of two K\"ahler metrics with negative upper bounds for holomorphic sectional curvatures has a negative upper bound for the holomorphic sectional curvature (see Lemma 2 in \cite{WU73}), it is enough to control the quantity of \eqref{eq:estimate_hsc} on $\mathcal C_\epsilon$. 
	
 Notice that the quasi-bounded geometry assumption of $\omega$ forces that $\|H_{\omega}\|_{\Omega}$ is bounded by some constant. Then $\mathcal C_{\epsilon}$ belongs to a compact subset of $ \Omega$ which does not touch $\partial \Omega$. Notice that $u$ can be zero on $\mathcal C_{\epsilon}$ only when $u$ takes (limit) values on $ \partial \Omega$. Then by the hypothesis on the holomorphic sectional curvature of $\omega$, $u$ on $\mathcal C_{\epsilon}$ is bounded below by a positive constant. Thus we can find sufficiently large $m\gg 0$ so that the right-hand side of \eqref{eq:estimate_hsc} becomes uniformly negative, and $H_{\omega_\Omega}$ has a negative upper bound for $m$ sufficiently large. 
	
	Next, we show that $\omega_{\Omega}$ admits quasi-bounded geometry by checking conditions in Definition \ref{def_quasi_bounded}. From the quasi-bounded geometry of $\omega$, the first requirement in Definition \ref{def_quasi_bounded} is clearly satisfied. Since the ball $D$ contains $\Omega$, the Poincar\'e metric $\omega_P$ on $\Omega$ is merely a weighted hermitian inner product in $\mathbb{C}^n$, and thus the second requirement is satisfied trivially. For the last requirement, $\omega$ satisfies the last requirement by the hypothesis, and $m\omega_p=m\sqrt{-1}\partial \overline{\partial} \log K_{P}$ never blows up with any $k$th-order derivative on $\Omega$, where $K_p$ is the Bergman kernel of $D$. This proves that $\omega_{\Omega}$ admits the
	quasi-bounded geometry. Consequently, by Lemma~\ref{ub_hsc} and Lemma~\ref{q-bounded}, $\omega_{\Omega}$ is uniformly equivalent to the Kobayashi--Royden metric. 
	
	Now, if $\omega$ is complete, $\omega_{\Omega}$ is also uniformly equivalent to the complete K\"ahler--Einstein metric of negative scalar curvature by Theorem 3 in \cite{DY17}. Moreover, by quasi-bounded geometry of $\omega$ and Lemma~\ref{q-bounded}, we have
	\begin{equation*}
		C^{-1} \, \chi_{\Omega} \leq \sqrt{\omega}
	\end{equation*}
	where $\chi_{\Omega}$ denotes the Kobayashi--Royden metric and $C$ is a universal constant. On the other hand, by the negative upper bound of $\omega_{\Omega}$ and Lemma~\ref{ub_hsc}, we also have 
	\begin{equation*}
		\sqrt{\omega_{\Omega}} \leq C \chi_{\Omega}.
	\end{equation*}
	Since $\omega\leq \omega_{\Omega}$ from the construction of $\omega_{\Omega}$, we obtain in all 
	\begin{equation*}
		C^{-1}  \, \chi_{\Omega} \leq \sqrt{\omega}  \leq \sqrt{\omega_{\Omega}} \leq C \chi_{\Omega}.
	\end{equation*}
	Hence $\omega$ is also uniformly equivalent to the Kobayashi--Royden metric.
	
	Otherwise, $\omega$ is incomplete, and the rest of conclusion follows from Proposition~\ref{thm:general_statement2} and Corollary~\ref{coro:incomplete}. 
	
\end{proof}

\begin{prop}\label{thm:general_statement2}
	Let $\Omega$ be a bounded weakly pseudoconvex domain in $\mathbb{C}^n$ with an incomplete K\"ahler metric $\omega$ satisfying all assumptions in Theorem~\ref{thm:general_statement}.	Then there is no complete K\"ahler metric with negatively pinched holomorphic sectional curvature. The same conclusion holds for the complete K\"ahler--Einstein metric of negative scalar curvature. 
	\begin{proof}
		By following the proof of Theorem~\ref{thm:general_statement}, an incomplete K\"ahler metric $g_{\Omega}$ with negatively pinched holomorphic sectional curvature is constructed. If a complete K\"ahler metric with negatively pinched holomorphic sectional curvature exists, this metric must be uniformly equivalent with the Kobayashi–Royden metric \cite[Theorem 2,3]{DY17}, and in particular, the Kobayashi–Royden metric must be complete, so it is a contradiction. In particular, for the complete K\"ahler-Einstein metric of negative scalar curvature, the holomorphic sectional curvature cannot be negatively pinched.
	\end{proof}

\end{prop}

\begin{proof}[Proof of Corollary~\ref{cor:new_class}]
	It follows from Theorem~\ref{thm:general_statement} that the holomorphic sectional curvature of $g_{\Omega}$ in \eqref{eq:Kahler-metric} is negatively pinched. Therefore, the second fundamental form of $S$ with respect to the restriction $g_{\Omega}|_S$ is bounded. By the decreasing property for holomorphic sectional curvature and the Gauss-Codazzi equation,  the holomorphic sectional curvature of $g_{\Omega}|_S$ is negatively pinched. The conclusion follows from Theorem 2 and Theorem 3 in \cite{DY17}.
\end{proof}

\begin{proof}[Proof of Corollary~\ref{cor:new_class2}]
	The proof is the same as the proof of Corollary 2 of \cite{SKY09}, as long as we confirm that the complete K\"ahler-Einstein metric has negative Ricci curvature and has bounded Riemannian sectional curvature. In the case of the complete K\"ahler-Einstein metric of negative scalar curvature based on Theorem~\ref{thm:general_statement}, it has quasi-bounded geometry, so it satisfies a stronger condition than the required condition \cite[Lemma 31]{DY17}.
\end{proof}

\section{Quasi-bounded geometry and Carath\'eodory--Reiffen metric}
Carath\'eodory completeness refers to the case where the Carath\'edory distance is complete as a topological distance \cite[Chapter 14]{JP13}. Even for bounded domains in $\mathbb{C}^n$, information about Carath\'eodory completeness and Carath\'eodory metric or distance is a very difficult question. In this section, we will discuss the relationship between quasi-bounded geometry and the Carath\'edory metric on K\"ahler manifolds.

\begin{lemma}\label{lemma:Caratheodory}
	Given $(M,\omega)$ a K\"ahler manifold, assume $\omega$ admits quasi-bounded geometry.  Then there exists $C>0$ such that
	\begin{equation*}
		\gamma_{M}(p;v)\leq C ||v||_{\omega} \text{ for any } p\in M, v\in T_p M. 
	\end{equation*}
	
	\begin{proof}
		For each $p\in M$, take any non-singular $\psi : B_{\mathbb{C}^n}(r) \rightarrow M$ by the definition of quasi-bounded geometry that maps $0$ to $p$. Then by the contraction (metric decreasing) property of the Carath\'eodory--Reiffen metric and the fact that $\gamma_{B_r(0)}$ coincides with the Poincar\'e metric of $B_{\mathbb{C}^n}(r)$, 
		\begin{align*}
			\gamma_{M}(p;v)&\leq \gamma_{B_{\mathbb{C}^n}(r)}(0;\psi^{*}v)\\
			&=\frac{1}{r}||\psi^{*}(v) ||_{\mathbb{C}^n}\\
			&\leq \frac{C}{r}||v||_{\omega}.
		\end{align*}
		In the last line, we use the requirement (2) of the definition of quasi-bounded geometry. We can replace $\frac{C}{r}$ by $C>0$ and the proof follows. 
	\end{proof}	
	
\end{lemma}

From this Lemma, we get one implication for the Carath\'eodory incompleteness .
\begin{coro}\label{coro:incomplete}
	Let $M$ be a complex manifold with an incomplete K\"ahler metric $\omega$ that admits quasi-bounded geometry.  Then $M$ is not Carath\'eodory complete. 
\end{coro}

\section{Discussion}
It would be very interseting to confirm whether an example corresponding to the class with incomplete K\"ahler metric $\omega$ in Theorem~\ref{thm:general_statement} exists. For example, by Corollary~\ref{coro:incomplete}, there is no incomplete K\"ahler metric admitting a quasi-bounded geometry in the case of Carath\'eodory complete domains.

\subsection*{Conflicts of interest}
The corresponding author states that there is no conflict of interest. 

\bibliographystyle{spmpsci}
\bibliography{reference}

\end{document}